\documentclass[a4paper]{article}

\usepackage{amsmath}
\usepackage{amssymb}
\usepackage{a4}

\def \sign{\mathop{\rm sign}\nolimits}

\def \dim{\mathop{\rm dim}\nolimits}

\def \im{\mathop{\rm Im}\nolimits}

\def \dis{\mathop{\rm Dis}\nolimits}

\def \C{{\mathbb C}}

\def \Z{{\mathbb Z}}
\def \N{{\mathbb N}}

\def \AA{{\cal A}}

\newtheorem{theorem}{Theorem}[section]

\newtheorem{lemma}[theorem]{Lemma}

\newtheorem{remarks}[theorem]{Remarks}

\newtheorem{definition}[theorem]{Definition}

\newenvironment{proof}
 {\begin{trivlist} \item[\hskip \labelsep {\bf Proof.}]}
 {\hfill$\Box$\end{trivlist}}
\newenvironment{lproof}[1]
 {\begin{trivlist} \item[\hskip \labelsep {\bf Proof (#1).}]}
 {\hfill$\Box$\end{trivlist}}

\begin{document}
\renewcommand{\theenumi}{(\roman{enumi})}
\normalsize

\title{Homotopy Type of Disentanglements of Multi-germs} 
\author{Kevin Houston \\
School of Mathematics \\ University of Leeds \\ Leeds, LS2 9JT, U.K. \\
e-mail: k.houston@leeds.ac.uk \\
http://www.maths.leeds.ac.uk/$\sim $khouston/
}
\date{\today }
\maketitle
\begin{abstract}
For a complex analytic map $f$ from $n$-space to $p$-space with $n<p$ and with an isolated instability at the origin, the disentanglement of $f$ is a local stabilization of $f$ that is analogous to the Milnor fibre for functions. 

For mono-germs it is known that the disentanglement is a wedge of spheres of possibly varying dimensions. In this paper we give a condition that allows us to deduce that the same is true for a large class of multi-germs.

\end{abstract}

\section{Introduction}
For a complex analytic map-germ $f:(\C ^n,S)\to (\C ^p,0)$ where $S$ is a finite set of points and where the origin in $(\C ^p,0)$ has an isolated instability we can find a nearby stable map which we can view as a stabilization of $f$. We call the discriminant of this stabilization the disentanglement of $f$. It is analogous to the Milnor fibre of an isolated complete intersection singularity. It is well known that such a Milnor fibre is homotopically equivalent to a wedge of spheres. In the case of disentanglements it is known that for $n\geq p-1$ that the disentanglement is homotopically a wedge of spheres of dimension $p-1$, see \cite{dm,vancyc}. (These references give the statements and proofs for mono-germs but the multi-germ proof is practically the same.)

For mono-germs with $n<p-1$ it was shown in \cite{loctop} that the disentanglement is homotopically a wedge of spheres but the spheres can be of different dimensions. An outstanding problem which seems to be much harder is to describe the topology for the multi-germ case. In this paper we show in Theorem~\ref{dis_homology} that the integer homology of the disentanglement is free abelian and hence one could conjecture that in analogy with the other cases that the disentanglement is homotopically equivalent to a wedge of spheres. Using a rather nice trick we show in Theorem~\ref{mainthm} that for a special but actually quite large class of maps this is indeed true.

\section{Disentanglements}
\label{sec:dis}
The details of the following definition of disentanglement can be found in \cite{mapfib}. The theorems there are stated for mono-germs but the extension to multi-germs is straightforward.
Suppose $S$ is a finite subset of $\C ^n$ and that $f:(\C ^n,S)\to (\C ^p,0)$, $n<p$, is a map-germ with an isolated instability at $0\in (\C^p,0)$ (equivalently, $f$ is finitely $\AA $-determined). Let $F:(\C ^n\times \C ^b,S\times 0)\to (\C ^p\times \C^b,0)$ be a versal unfolding, so that 
$F$ has the form $F(x,t)=(f_t(x),t)$. 
Let $\Sigma $ be the bifurcation space in the unfolding parameter space $\C ^b $, i.e., points such that the map $f_t:\C ^n \to \C ^p $ is unstable.
For corank 1 maps and maps in the nice dimensions the set $\Sigma $ is a proper subvariety of $\C ^b$ and so $\C^b \backslash \Sigma $ is connected. In other cases we relax our condition and ask that $f_t$ is topologically stable instead. In this case we also have that $\Sigma $ is a proper subvariety of $\C ^b$.

Given a Whitney stratification of the image of $f$ in $\C ^p $ there exists $\epsilon _0>0$ such that for all $0<\epsilon \leq \epsilon _0$, the real $(2p-1)$-sphere centred at $0$ of radius $\epsilon $ is transverse to the strata of the image of $f$.

Consider the map $f_t|f_t^{-1}(B_\epsilon )$ where $B _\epsilon $ is the closed ball of radius $\epsilon $ centred at $0$ with $\epsilon \leq \epsilon _0$ and $t\in \C^b \backslash \Sigma $. This stable map is called {\em{the disentanglement map}} of $f$ and its image 
is the {\em{disentanglement}} of $f$, denoted $\dis (f)$. This is independent of sufficiently small $\epsilon $ and $t$.
To ease notation we will write $\widetilde{f}$ rather than $f_t|f_t^{-1}(B_\epsilon )$. As $f^{-1}(0)$ is a finite set we can assume that $\widetilde{f}$ has the form $\widetilde{f}:\coprod_{j=1}^{|S|}  U_j \to \C ^p$ where $U_j$ is  a contractible open set.

In the case of $n\geq p$ a similar construction can be made where instead of the image of $f_t$ we use its discriminant.

The standard definition of the multiple point spaces of a map is the following.
\begin{definition}
Suppose that $f:X\to Y$ is a continuous map of topological spaces. Then 
the {\em{$k$th multiple point space of $f$}}, denoted $D^k(f)$, is the set 
\[
D^k(f):={\rm{closure}} \{ (x_1,\dots ,x_k)\in X^k | f(x_1)=\dots
= f(x_k), {\mbox{ such that }} x_i\neq x_j , i\neq j \} .
\] 
\end{definition}
The group of permutations on $k$ objects, denoted $S_k$, acts on $D^k(f)$
in the obvious way: permutation of copies of $X$ in $X^k$.

Let $d(f)=\sup \{ k \,  | \, D^k(\widetilde{f}) \neq \emptyset \}$ and let $s(f)$ be the number of branches of $f$ through the origin $0$ in $(\C ^p,0)$, i.e., the cardinality of $S$.

We can now generalize Corollary~4.8 of \cite{loctop} to the case of multi-germs. A version for rational cohomology for corank $1$ multi-germs was given in \cite{excellent}.
\begin{theorem}
\label{dis_homology}
Suppose that $f:(\C ^n,S)\to (\C ^p,0)$, $n<p$, has an isolated instability at $0\in (\C^p,0)$. Then, $\widetilde{H}_*(\dis (f);\Z )$ is free abelian and has non-trivial
groups possible only in dimensions $p-(p-n-1)k-1$ for all $2\leq k \leq d(f)$ and if $s(f)>d(f)$, then in
$\widetilde{H}_{d(f)-1}(\dis (f);\Z )$.

Furthermore, $H_0(\dis (f);\Z )=\Z$, i.e., $\dis (f)$ is connected.
\end{theorem}
If $p=n+1$, then a simple corollary of this is that the only non-trivial reduced homology groups 
occur in dimension $p-1$. This is known for mono-germs since the disentanglement is homotopically equivalent to a wedge of spheres, see \cite{vancyc}.

Suppose $X$ is a topological space with the homotopy type of a CW-complex
such that $S_k$ acts cellularly, that is, open cells go to open cells. 
Whitney stratified spaces can be triangulated so that they are CW-complexes. Furthermore, our multiple point spaces will be, generally speaking, Whitney stratified spaces such that $S_k$ acts cellularly on the CW-complex, see \cite{loctop}.  

Let $C_*(X;\Z )$ denote the cellular chain complex of $X$.
\begin{definition}
The alternating chain complex $C_*^{alt}(X;\Z )$ is the subcomplex of
 $C_*(X;\Z )$ given by
\[
C_n^{alt}(X;\Z ) := \{ c\in C_n(X;\Z ) | \sigma c=\sign (\sigma )c \mbox{ for all }
\sigma \in S_k \}.
\]
The alternating homology of $X$ is the homology of the complex 
$C_*^{alt}(X;\Z )$ and
is denoted $H^{alt} _*(X;\Z ) $.
\end{definition}
We can use the alternating homology of multiple point spaces to calculate the homology of the image of a finite and proper map.
\begin{theorem}
\label{icssthm}
Let $f:X\to Y$ be a finite and proper subanalytic map and let $Z$ be a (possibly empty) subanalytic subset of $X$ such that $f|Z$ is also proper. Then, there exists a spectral sequence
\[
E_1^{r,q}(f,f|Z)=  H^{alt}_q(D^{r+1}(f), D^{r+1}(f|Z);\Z ) \Rightarrow
H_*(f(X),f(Z);\Z ) ,
\]
where the differential is induced from the natural map 
$\varepsilon _{r+1,r}:D^{r+1}(f)\to D^r(f)$ given by 
$\varepsilon _{r+1,r} (x_1, x_2, \dots , x_r, x_{r+1}) = (x_1, x_2, \dots , x_r)$.
\end{theorem}
A proof is found in \cite{icss}. The precise details of the differential will not be required as our spectral sequences will be sparse. A spectral sequence for a single map rather than a pair also exists if we take $Z=\emptyset $.

Let $U$ and $W$ be open sets so that $F':U\to W$ is a one-parameter unfolding of $f$ of the form $F'(x,t)=(f_t(x),t)$ and $f_t$ gives the disentanglement of $f$ for $t\neq 0$.
\begin{lemma}
\label{E1_for_F}
The $E_1$ terms of the image computing spectral sequence of $F'$ are
\[
E_1^{r,q}(F') \cong 
\left\{ 
\begin{array}{ll}
\Z ^{\binom{s(f)}{r+1} }, & {\mbox{ for }} q= 0 {\mbox{ and }} 1\leq r+1 \leq s(f) , \\
0, & {\mbox{otherwise.}}
\end{array}
\right.
\]
The sequence collapses at $E_2$ and 
\[
E_{\infty }^{r,q} (F') \cong E_2^{r,q} (F') \cong
\left\{ 
\begin{array}{ll}
\Z & {\mbox{ for }} (r,q)= (0,0)  , \\
0 & {\mbox{ for }} (r,q)\neq (0,0).
\end{array}
\right.
\]
\end{lemma}
\begin{proof}
The proof is essentially the same as the proof of Lemma~3.3 of \cite{excellent}, only minor modifications need be made.
\end{proof}

We now prove an integer homology version of Lemma~3.4 of \cite{excellent} that does not assume corank $1$.
\begin{lemma}
\label{E1sparse}
Suppose that $f:(\C ^n,S)\to (\C ^p,0)$, $n<p$, has an isolated instability at $0$ with a one-parameter unfolding $F'$ and disentanglement map $\widetilde{f}$. 
As before let $d(f)=\sup \{ k \,  | \, D^k(\widetilde{f}) \neq \emptyset \}$ and $s(f)$ be the number of branches of $f$.

Then,
\[
E_1^{r,q} (F',\widetilde{f}) =
\left\{
\begin{array}{ll}
H^{alt}_{\dim_\C D^{r+1}(\widetilde{f})+1}(D^{r+1}(F'), D^{r+1}(\widetilde{f});\Z ) , & 
{\mbox{for }} q=\dim_\C D^{r+1}(\widetilde{f})+1\geq 0 \\
 & {\mbox{ and }}r+1\leq d(f),\\
\Z ^{\binom{s(f)}{r+1}} , & q=0 {\mbox{ and }} r+1>d(f), \\
0, & {\mbox{otherwise.}} 
\end{array}
\right.
\]
Here we define $\dim \emptyset =-1$. 
\end{lemma}
\begin{proof}
The proof is similar to the proof of Theorem ~4.6 of \cite{loctop}. The main difference is that $D^k(F')$ is no longer connected and in general neither is $D^k(\widetilde{f})$. This means we have to be careful about the bottom row of the spectral sequence, i.e., $E^{r,0}_1$.

By reasoning similar to the proof of Theorem~4.6 in \cite{loctop} we have that 
\[
H^{alt}_i(D^k(F'), D^k(\widetilde{f});\Z )=0 {\text{ for }} i \leq nk-p(k-1) = \dim_\C D^k(\widetilde{f}) . 
\]
The main change in the proof is that the fibration $g$ is a multi-germ fibration and we take the alternating homology of the fibres of this map.

As in Theorem~4.6 of \cite{loctop} we have that $H^{alt}_i(D^k(\widetilde{f});\Z )=0$ for $i>\dim D^k(\widetilde{f}) $ and is free abelian for $i= D^k(\widetilde{f}) $. 

Therefore for $r+1\leq d(f)$ we have that 
\[
E^{r,q}_1(F',\widetilde{f} ) = 
H^{alt}_{\dim_\C D^{r+1}(\widetilde{f})+1}(D^{r+1}(F'), D^{r+1}(\widetilde{f});\Z ) {\text{ for }} q=\dim_\C D^{r+1}(\widetilde{f})+1
\]
and zero otherwise. If $r+1>d(f)$, then $D^{r+1}(\widetilde{f})$ is empty by definition of $d(f)$ and so 
\[
E^{r,q}_1(F',\widetilde{f} ) \cong E^{r,q}_1(F' ) 
\]
for all $q$. Therefore by Lemma~\ref{E1_for_F} we have 
\[
E^{r,q}_1(F',\widetilde{f} ) =
\Z ^{\binom{s(f)}{r+1}} 
\]
for $q=0$ and zero otherwise.
\end{proof}
We are now in a position to prove Theorem~\ref{dis_homology}.
\begin{lproof}{of Theorem~\ref{dis_homology}}
From Lemma~\ref{E1_for_F} and Lemma~\ref{E1sparse} we can see that the bottom row of the image computing spectral sequence for $(F',\widetilde{f})$ is exact except possibly at $E^{d(f),0}_1$. Hence,
$E^{r,0}_2(F',\widetilde{f})=0$ except possibly for $E^{d(f),0}(F',\widetilde{f})$ and this will be free abelian.

From Lemma~\ref{E1sparse} we can see that there are no other non-trivial differentials and so 
$E^{r,q}_2(F',\widetilde{f})\cong E^{r,q}_1(F',\widetilde{f})$ for all $r$ and all $q\neq 0$.

From the positions of the non-trivial groups in $E^{r,q}_2(F',\widetilde{f})$ we can see that the sequence collapses at $E_2$. Since there is no torsion there are no extension problems and we can deduce that 
\[
H_*(\im F', \im \widetilde{f} ;\Z ) ,
\]
where $\im $ denotes image, has non-trivial groups possible only in dimensions $p-(p-n-1)k$ for $2\leq k \leq d(f)$ and in  dimension $d(f)$ if $s(f)>d(f)$.

Since the image of $F'$ has the homology of a point by Lemma~\ref{E1_for_F} we deduce that 
$H_*(\dis (f);Z ) = H_*(\im \widetilde{f} ;\Z )$ has the homology described in the statement of the theorem.
\end{lproof}

\section{The homotopy type of the disentanglement}
\label{sec:htpy}
In this section we use a little trick to show that in a large number of cases the homotopy type of the disentanglement of a multi-germ with $n<p$ is homotopically a wedge of spheres of varying dimensions.

\begin{theorem}
\label{mainthm}
Suppose that $f:(\C ^n,S)\to (\C ^p,0)$, $n<p$, is a multi-germ with an isolated instability 
at $0$ and $s(f)\leq d(f)$. Then $\dis (f)$ is homotopically equivalent
to a wedge of spheres where the possible (real) dimensions are $p-(p-n-1)k-1$ for all $2\leq k \leq d(f)$.
\end{theorem}
To prove this theorem we shall construct a map, the image of which is homotopically equivalent to the disentanglement of $f$ and such that its image computing spectral sequence is much simpler.
\begin{lemma}
\label{findy}
Suppose that $f:(\C ^n,S)\to (\C ^p,0)$, $n<p$, is a multi-germ with an isolated instability 
at $0$ and that $\widetilde{f}:\coprod_{j=1}^{s(f)}  U_j \to \C ^p$ is a stabilization giving the disentanglement.

If $s(f)\leq d(f)$, then there exists $y\in \C ^p$ such that $\widetilde{f}^{-1}(y)\cap U_j\neq \emptyset $ for all $j=1, \dots , s(f)$.
\end{lemma}
\begin{proof}
As before suppose that $F'$ is one-parameter unfolding of $f$. By Lemma~\ref{E1sparse} we know that 
\[
H^{alt}_0(D^k(F'), D^k(\widetilde{f});\Z )=0 
\]
for all $k\leq d(f)$. In particular, $H^{alt}_0(D^{s(f)}(F'), D^{s(f)}(\widetilde{f});\Z )=0$.
From Lemma~\ref{E1_for_F} we know that $H^{alt}_0(D^{s(f)}(F');\Z )=E^{s(f)-1,0}_1(F')=\Z $.
This implies that the natural inclusion of $D^{s(f)}(\widetilde{f})$ into $D^{s(f)}(F')$ induces a surjection
\[
H^{alt}_0(D^{s(f)}(\widetilde{f} );\Z ) \to H^{alt}_0(D^{s(f)}(F');\Z )
\]
(which in fact is an isomorphism if $\dim D^{s(f)}(\widetilde{f} ) >0$).

Thus there exists a point $z$ in $D^{s(f)}(\widetilde{f} )$ so that its orbit produces a generator of 
$H^{alt}_0(D^{s(f)}(\widetilde{f} );\Z )$ which maps to the generator of 
$H^{alt}_0(D^{s(f)}(F');\Z )$. We can choose this $z$ so that it is not in the image of the map $\varepsilon _{s(f)+1,s(f)}$ defined in Theorem~\ref{icssthm}.
This is because this image will be of smaller dimension and hence will not disconnect $D^{s(f)}(\widetilde{f})$. 

If $z=\left( z_1, z_2, \dots ,z_{s(f)} \right) \in \left( \coprod_{l=1}^{s(f)}  U_l \right) ^{s(f)}$, then as the (alternating) orbit of this point generates  $H^{alt}_0(D^{s(f)}(F');\Z )$ we must have that each $z_j$ and $U_l$ can be uniquely paired. Since $z\in D^{s(f)}(\widetilde{f})$, there exists a $y\in \C ^p$ such that $\widetilde{f}(z_j)=y$ for $1\leq j \leq s(f)$.
In other words we have shown that there exists $y\in \C ^p$ such that $\widetilde{f}^{-1}(y)\cap U_j\neq \emptyset $ for all $j=1, \dots , s(f)$.
\end{proof}

For a map $g:X\to Y$, let $M_k(g)$ be the image of the map from $D^k(g)$ to $Y$ given by $(x_1,\dots , x_k)\mapsto f(x_1)$.

Consider $F'$ the one-parameter unfolding of $f$ and let $f_t$ be a family of maps so that $f_0=f$ and  $f_t$ is a disentanglement map for $0<t\leq 1$.  
The condition $s(f)\leq d(f)$ in the statement of the theorem implies via the preceding lemma that $M_{s(f)}(\widetilde{f})$ is non-empty as it contains $y$. (Note that the proof gives that $y\in M_{s(f)}(\widetilde{f})\backslash M_{s(f)+1}(\widetilde{f})$.) This implies that $M_{s(f)}(f_t)$ must be non-empty for $t\neq 0$ also. Thus $M_{s(f)}(F')$ is at least one dimensional and as it is an analytic subspace of $\C ^p \times \C ^b$ it is path connected. Since this set must pass through the origin in $\C ^p \times \C ^b$ and from the path connectedness we know that  there exists a path $\alpha :[0,1]\to M_{s(f)}(F')$ such that 
$\alpha (0)=0$, 
$\alpha (t)\in M_{s(f)}(f_t)\backslash M_{s(f)+1}(f_t)$, and 
$\alpha (1)=y\in M_{s(f)}(\widetilde{f})\backslash M_{s(f)+1}(\widetilde{f})$.

Let $Y=\dis (f)$ and $Y^+ = (Y\coprod [0,1]) /\sim $
where $\sim $ is the identification of $y\in Y$ and $1\in [0,1]$.
Then obviously $Y$ is homotopically equivalent to $Y^+$ as we have merely attached
an interval to $Y$. 

We shall attach intervals to each $U_j$ and then 
identify their free ends to a single point. Let $x_j$ be a point $\widetilde{f}^{-1}(y)\cap U_j$ that can be used in generating $H^{alt}_0(D^{s(f)}(\widetilde{f} );\Z )$ as above and
let $X^+ _j= (U_j \coprod [0,1])/(x_j \sim 1)$.
Next, let $X^+ =(\coprod X^+_j )/\sim $, where
$\sim $ is the identification of all the origins of the intervals
in the copies of $[0,1]$. 
This construction is such that $X^+$ is homotopically equivalent to a wedge of all the $U_j$.

Via straightforward inclusion we can consider $X$, $U_j$, etc.\ as subsets of $X^+ $.
Define $f^+ :X^+ \to Y^+ $ by 
\[
\left\{
\begin{array}{l}
f^+ |\coprod U_j := \widetilde{f} \\
f^+ |X^+_j \backslash U_j := \alpha (t) {\text{ for all }} j=1,\dots , s(f) {\text{ and }} t\in [0,1] .
\end{array}
\right.
\]
It is easy to see that $f^+$ is a continuous map defined at all points of
$X^+$ and its image is $Y^+$.

\begin{lemma}
\label{mainlemma}
The set $Y^+$ is homotopically equivalent to a wedge of spheres with possible dimensions as in Theorem~\ref{mainthm}.
\end{lemma}
\begin{proof}
Although the statement of Theorem~\ref{icssthm} was for subanalytic maps the image computing spectral sequence exists for $f^+$ since $\widetilde{f}$ is a complex analytic map and the extra bits we add on to get $f^+$ are particularly simple, see \cite{icss} for details.
That is, there exists a spectral sequence,
\[
E_1^{r,q} (f^+) = H^{alt}_q(D^{r+1}(f^+);\Z ) \implies 
H_*(Y^+;\Z ).
\]
We will show that the $E^1$ page of this sequence is particularly sparse and we
shall compare it with the image computing spectral sequence for $\widetilde{f}$.

The obvious inclusions $h:\widetilde{Y} \to Y^+$ and $h_1:\widetilde{X} \to X^+$ lead to inclusions $h_k:D^k(\widetilde{f}) \to D^k(f^+)$ for all $k\geq 2$.

As $\widetilde{f}$ and $F$ are complex analytic, the multiple point spaces $D^k(\widetilde{f})$ are Whitney stratifiable and hence triangulable. This triangulation can be chosen such that we can construct $D^k(f^+)$ by attaching a number of $1$-cells to $0$-cells in $D^k(\widetilde{f})$ and then identifying the ends of the $1$-cells to a single $0$-cell which is invariant under the action of $S_k$. (Note that the action of $S_k$ on the interior of the $1$-cells is the same as the one on the $0$-cells in $D^k(\widetilde{f})$, i.e., the ends of the $1$-cells that are not fixed by the $S_k$-action.) 
For each $k$ define $V_k$ to be the closed cellular complex given by these $1$-cells and $0$-cells.

In summary we have $D^k(f^+) = D^k(\widetilde{f}) \cup V_k$ where $D^k(\widetilde{f}) \cap V_k$ is a finite collection of $0$-cells. Call this intersection $P_k$. (Essentially these cells will arise from the orbit of $k$-tuples of the points $z_j$ in the proof of Lemma~\ref{findy}.) 

We can use an alternating homology version of the Mayer-Vietoris sequence to get a long exact sequence:
\begin{equation*}
\begin{split}
 & \dots
\rightarrow H^{alt}_n(D^k(\widetilde{f}) \cap V_k ;\Z )
\rightarrow H^{alt}_n(D^k(\widetilde{f}) ;\Z ) \oplus H^{alt}_n( V_k ;\Z ) \qquad \\
&\qquad \qquad  \rightarrow H^{alt}_n(D^k(f^+);\Z )
\rightarrow H^{alt}_{n-1}(D^k(\widetilde{f}) \cap V_k ;\Z )
\rightarrow \dots 
\end{split}
\end{equation*}
The alternating homology of $V_k$ is zero for all $k>1$. This is because obviously no $1$-cell can be a cycle and because all the $0$-cells are homologous to the $S_k$-invariant $0$-cell. This latter fact implies that the zeroth alternating homology group is zero by Lemma~2.6(iii) of \cite{loctop}.

Therefore the Mayer-Vietoris sequence gives the long exact sequence
\[
\dots \rightarrow H^{alt}_n(P_k ;\Z )
\rightarrow H^{alt}_n(D^k(\widetilde{f}) ;\Z )
\rightarrow H^{alt}_n(D^k(f^+);\Z )
\rightarrow H^{alt}_{n-1}(P_k ;\Z )
\rightarrow \dots 
\]
The map $H^{alt}_0(P_k ;\Z )\rightarrow H^{alt}_0(D^k(\widetilde{f}) ;\Z )$ is injective since the points in $P_k$ were chosen to be generators of $H^{alt}_0(D^k(\widetilde{f}) ;\Z )$. (If $D^k(\widetilde{f})$ is not $0$-dimensional, then it is in fact a bijection.)

Therefore we conclude that for $k\geq 2$.
\[
H^{alt}_n(D^k(f^+);\Z )= \left\{ 
\begin{array}{ll}
H^{alt}_n(D^k(\widetilde{f}) ;\Z ) ,& {\text{for }} n = \dim_\C  D^k(\widetilde{f}) , \\
0, & {\text{otherwise}} .
\end{array}
\right.
\]
If $k=1$, then  $D^k(f^+)$ is homotopically equivalent to the wedge of the $U_j$ and so $H_*(D^k(f^+);\Z )$ is just the ordinary homology of a point.

From this we deduce that 
\[
E_1^{r,q} (f^+) = \left\{ 
\begin{array}{ll}
\Z ^{\mu _r}  ,& {\text{for }} q = \dim_\C  D^{r+1}(\widetilde{f}) , r>1,{\text{ some }} \mu _r \in \N \cup \{ 0 \} ,\\
\Z & {\text{for }} (r,q)=(0,0) \\
0, & {\text{otherwise}} .
\end{array}
\right.
\]
Therefore the image computing spectral sequence for $f^+$ degenerates at the first page, i.e., $E_m^{r,q}(f^+) = E_1^{r,q}(f^+)$ for all $m\in\N $. Since the groups are free abelian we can read off the homology of $Y^+$ from this sequence, noting that the possible dimensions with non-zero homology are those as in Theorem~\ref{mainthm}.

Then using the same reasoning as in Theorem~4.24 of \cite{loctop} for the mono-germ case we can deduce that $Y^+$ is homotopically equivalent to a wedge of spheres. Note that, if $Y$ has non-trivial homology in dimension $1$, then this implies that $Y$ is homotopically a wedge of circles and if $n<p-1$, then this is the only situation where $Y$ is not simply connected. (If $n=p-1$, then we already know from \cite{vancyc} that the disentanglement is homotopically a wedge of spheres.)
\end{proof}

We now complete the proof of the theorem.
\begin{lproof}{of Theorem~\ref{mainthm}}
All we have to do now is show that $Y$ and $Y^+$ have the same homotopy type. To do this we shall compare the spectral sequences for $f^+$ and $\widetilde{f}$.

We saw in the preceding lemma that the natural inclusion $h_k:D^k(\widetilde{f}) \to D^k(f^+)$ induces an isomorphism 
\[
H^{alt}_n(D^k(\widetilde{f});\Z ) \to H^{alt}_n( D^k(f^+);\Z  )
\]
for all $n\geq 1$ and $k\geq 1$. 

From the proof of Lemma~\ref{E1sparse} we can see that the bottom row of the spectral sequence for $\widetilde{f}$ may have non-zero terms, i.e., $E_1^{r,0}(\widetilde{f})\cong E_1^{r,0}(F')$ for $r+1\leq d(f)$ and $E_1^{r,0}(F')$ may be non-zero by Lemma~\ref{E1_for_F}. However, on the second page of the spectral sequence we have $E_1^{0,0}(\widetilde{f})=\Z $ and only one other possible non-zero term at $E_1^{s(f),0}(\widetilde{f})$ (when $s(f)>d(f)$), see the proof of Theorem~\ref{dis_homology}

Therefore the map $E_2^{r,q} (\widetilde{f}) \to E_2^{r,q} (f^+)$ is an isomorphism for all $r$ and $q$.  
Since the sequence for $f^+$ degenerates at the first page and the one for $\widetilde{f}$ degenerates at the second page this implies that the natural map $h:Y\to Y^+$ induces an isomorphism on homology. 

Now, by Proposition~4.21 of \cite{loctop} we have that $Y$ is simply connected. (The statement there is for a complex analytic map - a condition that a wayward Latex macro changed to a German double S.)  By Lemma~\ref{mainlemma} and the comment at the end of its proof we have that $Y^+$ is simply connected or is homotopically a wedge of circles. Therefore, in the former case, by Whitehead's Theorem (\cite{switzer} p 187) the map $h$ induces an isomorphism on all homotopy groups. Since both spaces are triangulable this implies that they are homotopically equivalent, (\cite{switzer} p 187). In the latter we note from the construction of the spectral sequence, see \cite{go} or \cite{icss}, that $Y$ is homotopically equivalent to a space constructed from the contractible $U_j$ by adding $1$-cells that correspond to double points. Hence $Y$ is also homotopically a wedge of circles. Now we merely note that this number of circles is the same as for $Y^+$ as $h$ induces an isomorphism on homology. Hence $Y$ and $Y^+$ are homotopically equivalent. 
\end{lproof}

\begin{remarks}
\begin{enumerate}
\item The class of germs in the theorem with $s(f)\leq d(f)$ is very large. Obviously, the condition trivially holds for mono-germs. Now consider a bi-germ $f$ with branches $f_1$ and $f_2$. If any branch is not an immersion, say $f_1$, then $D^2(f_1)\neq \emptyset $. Now if the dimension of $D^2(f)$ is greater than $1$, then as $f_1$ and $f_2$ will generally meet transversally (outside the origin), the intersection of the double points of $f_1$ and the image of $f_2$ will be non-empty. Thus we have a triple point. Again depending on dimension, this means that the disentanglement has a triple point. That is, $d(f)\geq 3$. Similar reasoning holds for any $s(f)$.

\item We can construct maps with isolated instabilities such that $s(f)\geq d(f)$ and their difference is any arbitrary number. Consider $s$ lines in the complex plane all meeting at the origin. This is the image of a map with an isolated instability. The only type of singularity occurring in the disentanglement is an ordinary double point. Hence $d(f)=2$ but $s(f)=s$ is arbitrary.

Note though that in this case the disentanglement is still homotopically equivalent to a wedge of spheres.

\item If $s(f)>d(f)$, then Theorem~\ref{dis_homology} gives us the hope that the disentanglement is homotopically a wedge of spheres. As an example consider the ordinary quadruple surface point in three-space, that is, the image of four $2$-planes into $\C ^3$ such that the intersection of planes is pairwise transverse. It is easy to calculate that $s(f)=4$ but that $d(f)=3$ and that the disentanglement is homotopically equivalent to a single $2$-sphere. 

\item If we have a sufficiently generic finite complex analytic map from a complete intersection of dimension $n$ to $\C ^p$, $n<p$, then it is possible to show that the complex links of strata have free abelian integer homology, see Corollary~4.17 of \cite{loctop}. It should be possible to use reasoning similar to the proof of Theorem~\ref{mainthm} to show that a large class of the complex links are homotopically equivalent to a wedge of spheres of varying dimensions.
\end{enumerate}
\end{remarks}

\end{document}